\documentclass{amsart}

\newcommand{\R}{\mathbb{R}}

\newcommand{\Z}{\mathbb{Z}}
\newcommand{\D}{\mathbb{D}}
\newcommand{\T}{\mathbb{T}}
\newcommand{\De}{\mathcal{D}}
\newcommand{\sgn}{\operatorname{sgn}}
\newcommand{\pv}{\operatorname{p.v.}}

\newtheorem{thm}{Theorem}
\newtheorem{prop}{Proposition}
\newtheorem{lemma}{Lemma}

\newtheorem{conj}{Conjecture}

\theoremstyle{definition}

\theoremstyle{remark}
\newtheorem{remark}{Remark}

\begin{document}

\title[A computation of Poisson kernels for biharmonic operators]{A computation 
of Poisson kernels for\\ some standard weighted biharmonic operators\\ 
in the unit disc}

\date{\today}

\author{Anders Olofsson}

\address{Falugatan 22 1tr\\ SE-113 32 Stockholm\\ Sweden}

\email{ao@math.kth.se}



\begin{abstract}
We compute Poisson kernels for  
integer weight parameter standard weighted biharmonic operators
in the unit disc with Dirichlet boundary conditions.
The computations performed extend the supply of explicit examples 
of such kernels and suggest similar 
formulas for 
these Poisson kernels to hold true in more generality. 
Computations have been carried out using the open source 
computer algebra package Maxima.
\end{abstract}

\maketitle

\setcounter{section}{-1}

\section{Introduction}

\setcounter{equation}{0}
\setcounter{thm}{0}
\setcounter{prop}{0}
\setcounter{lemma}{0}
\setcounter{cor}{0}
\setcounter{remark}{0}

We address in this paper the problem of finding explicit formulas for Poisson kernels 
for weighted biharmonic operators of the form 
$\Delta w^{-1}\Delta$
in the unit disc $\D$ with Dirichlet boundary conditions, 
where $\Delta=\partial^2/\partial z\partial\bar z$, $z=x+iy$, 
is the Laplacian in the complex plane and $w=w_\gamma$ is a 
weight function of the form
$$
w_\gamma(z)=(1-\lvert z\rvert^2)^\gamma,\quad z\in\D,
$$ 
for some real parameter $\gamma>-1$. 
Such a weight function $w_\gamma$ is commonly referred to as a standard weight.

Let us first describe the context of these Poisson kernels. 
Let $w:\D\to(0,\infty)$ be a smooth radial weight function 
and consider the 
weighted biharmonic Dirichlet problem
\begin{equation}\label{Dirichletproblem}
\left\{
\begin{array}{rclc}
\Delta w^{-1}\Delta u&=&0& \text{in}\ \D,\\
u&=&f_0&\text{on}\ \T,\\
\partial_n u&=&f_1&\text{on}\ \T.
\end{array}
\right.
\end{equation}
Here $\T=\partial\D$ is the unit circle and 
$\partial_n$ denotes differentiation in the inward normal direction.
The first equation in (\ref{Dirichletproblem}), 
the biharmonic equation $\Delta w^{-1}\Delta u=0$, 
is evaluated in the distributional sense and defines 
a class of functions which we call $w$-biharmonic.    
In full generality 
the boundary datas $f_j\in\De'(\T)$ ($j=0,1$) are distributions 
on $\T$ 
and the boundary conditions in (\ref{Dirichletproblem}) 
are interpreted in a distributional sense as follows:
Let $u$ be a smooth function in $\D$ and let $f_0\in\De'(\T)$. 
We say that $u=f_0$ on $\T$ in the distributional sense if 
$\lim_{r\to1}u_r=f_0$ in $\De'(\T)$, where 
\begin{equation}\label{dfndilation}
u_r(e^{i\theta})=u(re^{i\theta}),\quad e^{i\theta}\in\T,
\end{equation}
for $0\leq r<1$.
Similarly, the inward normal derivative $\partial_n u$ of $u$ 
is defined by 
$$
\partial_n u=\lim_{r\to1}(u_r-f_0)/(1-r)\quad\text{in}\ \De'(\T)
$$
provided the limit exists,
where $u=f_0$ on $\T$ in the distributional sense. 
In an earlier paper~\cite{Olofsson05} we have shown that 
if the weight function $w$ is area integrable and 
has enough mass near the boundary, then the 
distributional Dirichlet problem (\ref{Dirichletproblem}) 
has a unique solution solution $u$, which has the representation 
\begin{equation}\label{representationformula}
u(z)=(F_{w,r}*f_0)(e^{i\theta})+(H_{w,r}*f_1)(e^{i\theta}),
\quad z=re^{i\theta}\in\D,
\end{equation}
in terms of two functions $F_w$ and $H_w$; 
here $F_{w,r}=(F_w)_r$ in accordance with \eqref{dfndilation}
and similarly for $H_w$. 
The symbol $*$ denotes convolution  of distributions on $\T$.
The Poisson kernels $F_w$ and $H_w$ are characterized  by the  problems 
$$
\left\{
\begin{array}{rclc}
\Delta w^{-1}\Delta F_w&=&0&\text{in}\ \D,\\
F_w&=&\delta_1&\text{on}\ \T,\\
\partial_n F_w&=&0&\text{on}\ \T,
\end{array}
\right.
\quad\text{and}\quad
\left\{
\begin{array}{rclc}
\Delta w^{-1}\Delta H_w&=&0&\text{in}\ \D,\\
H_w&=&0&\text{on}\ \T,\\
\partial_n H_w&=&\delta_1 &\text{on}\ \T,
\end{array}
\right.
$$
interpreted in the above distributional sense; here  
$\delta_{e^{i\theta}}$ denotes the unit Dirac mass at $e^{i\theta}\in\T$.

The above representation formula \eqref{representationformula}
applies to all weight functions of the form $w=w_\gamma$, where $\gamma>-1$, 
and we denote by
$F_\gamma=F_{w_\gamma}$ and $H_\gamma=H_{w_\gamma}$ 
the corresponding Poisson kernels for \eqref{Dirichletproblem} described above. 
Some formulas for $F_\gamma$ and $H_\gamma$ and are known.  
In the simplest classical unweighted case ($\gamma=0$) these kernels 
are given by 
$$
F_0(z)=\frac{1}{2}\frac{(1-\lvert z\rvert^2)^2}{\lvert1-z\rvert^2}
+\frac{1}{2}\frac{(1-\lvert z\rvert^2)^3}{\lvert1-z\rvert^4},
\quad z\in\D,
$$
and 
$$
H_0(z)=\frac{1}{2}\frac{(1-\lvert z\rvert^2)^2}{\lvert1-z\rvert^2},
\quad z\in\D,
$$
respectively (see~\cite{AH}). The functions $F_1$ and $H_1$ are also known 
and given by the formulas 
$$
F_1(z)=
\frac{1}{2}\frac{(1-\vert z\vert^2)^3}{\vert 1-z\vert^2}
+\frac{(1-\vert z\vert^2)^4}{\vert 1-z\vert^4}
-\frac{(1-\vert z\vert^2)^3}{\vert 1-z\vert^4}
+\frac{1}{2}\frac{(1-\vert z\vert^2)^5}{\vert 1-z\vert^6},
\quad z\in\D,
$$
and
$$
H_1(z)=
\frac{1}{2}\frac{(1-\vert z\vert^2)^3}{\vert 1-z\vert^2}+
\frac{1}{4}\frac{(1-\vert z\vert^2)^4}{\vert 1-z\vert^4},
\quad z\in\D
$$
(see~\cite{Olofsson05}). 
The next generation of Poisson kernels 
$F_2$ and $H_2$ is given by 
\begin{align*}
F_2(z)=&\frac{1}{2}\frac{(1-\vert z\vert^2)^4}{\vert1-z\vert^2}
+\frac{3}{2}\frac{(1-\vert z\vert^2)^5}{\vert1-z\vert^4}
-\frac{3}{2}\frac{(1-\vert z\vert^2)^4}{\vert1-z\vert^4}\\
&+\frac{3}{2}\frac{(1-\vert z\vert^2)^6}{\vert1-z\vert^6}   
-\frac{3}{2}\frac{(1-\vert z\vert^2)^5}{\vert1-z\vert^6}
+\frac{1}{2}\frac{(1-\vert z\vert^2)^7}{\vert1-z\vert^8},
\quad z\in\D,
\end{align*}
and
$$
H_2(z)=\frac{1}{2}\frac{(1-\vert z\vert^2)^4}{\vert1-z\vert^2}
+\frac{1}{2}\frac{(1-\vert z\vert^2)^5}{\vert1-z\vert^4}
-\frac{1}{4}\frac{(1-\vert z\vert^2)^4}{\vert1-z\vert^4}
+\frac{1}{6}\frac{(1-\vert z\vert^2)^6}{\vert1-z\vert^6},
\quad z\in\D
$$
(see Section~\ref{F2H2computation}). 

We have performed further calculations of Poisson kernels 
$F_\gamma$ and $H_\gamma$ along these lines.
We propose to write the third generation Poisson kernel $F_3$ as 
$$
F_3(z)=\sum_{\beta=1}^5\frac{f_\beta(\vert z\vert^2)}{\vert1-z\vert^{2\beta}},
\quad z\in\D,
$$
where $\{f_\beta\}_{\beta=1}^5$ are polynomials given by 
\begin{align*}
2f_1(x)&=(1-x)^5,\\
2f_2(x)&=4(1-x)^6-4(1-x)^5,\\
2f_3(x)&=6(1-x)^7-8(1-x)^6+2(1-x)^5,\\
2f_4(x)&=4(1-x)^8-4(1-x)^7,\\
2f_5(x)&=(1-x)^9.
\end{align*}
Notice that the columns in the above table of coefficients 
are up to normalization rows in the standard table of binomial coefficients 
(Pascal's triangle) and that each $f_\beta$ is a linear combination of 
monomials $(1-x)^k$ with exponents $k$ in the range 
$\max(2\beta-1,\gamma+2)\leq k\leq\beta+\gamma+1$; here $\gamma=3$.
The normalization factors are determined by $f_2(0)=f_3(0)=f_4(0)=0$.
We conjecture this structure to prevail for all 
integer weight powers $\gamma\geq0$ (see Conjecture~\ref{Fform}).
Similarly, the function $H_3$ has the form 
$$
H_3(z)=\sum_{\beta=1}^4\frac{h_\beta(\vert z\vert^2)}{\vert1-z\vert^{2\beta}},
\quad z\in\D,
$$
where $\{h_\beta\}_{\beta=1}^4$ are polynomials given by 
\begin{align*}
2h_1(x)&=(1-x)^5,\\
4h_2(x)&=3(1-x)^6-2(1-x)^5,\\
6h_3(x)&=3(1-x)^7-2(1-x)^6,\\
8h_4(x)&=(1-x)^8,
\end{align*}
and we conjecture similar formulas to hold true for $H_\gamma$ 
with $\gamma$ a non-negative integer (see Conjecture~\ref{Hform}).

To further illustrate Conjectures~\ref{Fform} and~\ref{Hform} 
we include here also formulas for the fourth generation 
Poisson kernels $F_4$ and $H_4$: 
$$
F_4(z)=\sum_{\beta=1}^6\frac{f_\beta(\vert z\vert^2)}{\vert1-z\vert^{2\beta}},
\quad z\in\D,
$$
where 
\begin{align*}
2f_1(x)&=(1-x)^6,\\
2f_2(x)&=5(1-x)^7-5(1-x)^6,\\
2f_3(x)&=10(1-x)^8-15(1-x)^7+5(1-x)^6,\\
2f_4(x)&=10(1-x)^9-15(1-x)^8+5(1-x)^7,\\
2f_5(x)&=5(1-x)^{10}-5(1-x)^9,\\
2f_6(x)&=(1-x)^{11},
\end{align*}
and 
$$
H_4(z)=\sum_{\beta=1}^5\frac{h_\beta(\vert z\vert^2)}{\vert1-z\vert^{2\beta}},
\quad z\in\D,
$$
where 
\begin{align*}
2h_1(x)&=(1-x)^6,\\
4h_2(x)&=4(1-x)^7-3(1-x)^6,\\
6h_3(x)&=6(1-x)^8-6(1-x)^7+(1-x)^6,\\
8h_4(x)&=4(1-x)^9-3(1-x)^8,\\
10h_5(x)&=(1-x)^{10}.
\end{align*}

In total we have verified Conjectures~\ref{Fform} and~\ref{Hform}
for integer weight parameters $\gamma$ 
in the range $0\leq\gamma\leq80$ using computer calculations 
(see Section~\ref{expectedFHformulas}).
In this way we have now available explicit formulas for 
$F_\gamma$ and $H_\gamma$ for considerably more weight parameters 
than was previously known.

Our initial method of calculation of $F_\gamma$ and $H_\gamma$'s 
breaks down into two parts and might be of interest in other contexts as well: 
the construction of nontrivial $w_\gamma$-biharmonic functions with 
prescribed $\delta_1$-type singularities on 
the boundary $\T=\partial\D$ and then to compute 
the appropriate normalizations of the functions constructed.
The nontrivial $w_\gamma$-biharmonic functions are found using 
a method of decoupling of Laplacians described in 
Section~\ref{decouplingLaplacians}. 
The appropriate normalizations are found analyzing 
the distributional boundary behavior
of building block functions of the form 
$u(z)=(1-\lvert z\rvert^2)^{2\beta-1}/\lvert1-z\rvert^{2\beta}$ 
which we discuss in 
Section~\ref{boundarybehaviorbbfcns}.
In Section~\ref{F2H2computation} 
we compute $F_2$ and $H_2$ using this method of calculation.

We wish to mention also that Conjectures~\ref{Fform} and~\ref{Hform}
imply certain $L^1$-bounds for $F_\gamma$ and $H_\gamma$. 
These $L^1$-bounds have as consequence a regularity property that 
the Dirichlet problem \eqref{Dirichletproblem} 
can be solved for boundary data $f_j\in B_j$ ($j=0,1$)
in certain admissible pairs of  homogeneous Banach spaces $B_j$ ($j=0,1$) 
with boundary values evaluated in the norms of these spaces 
$$
\lim_{r\to1}u_r=f_0\quad\text{in}\ B_0\quad\text{and}\quad
\lim_{r\to1}(u_r-f_0)/(1-r)=f_1\quad\text{in}\ B_1
$$
(see Section~\ref{furtherresults}). 

The keyword fourth order equations brings to mind 
a mathematical study of properties of materials. 
A different depart of interest for study of weighted biharmonic operators 
of the form $\Delta w^{-1}\Delta$ comes from Bergman space theory 
where related potential theory 
has been used to study properties 
of factorization and approximation of analytic functions subject to 
area integrability constraints (see~\cite{ARS,DKSS,HKZ,O04}). 
A significant contribution in this direction 
is the paper Hedenmalm, 
Jakobsson and Shimorin~\cite{HJS} on biharmonic maximum principles; 
predecessors of this work are Hedenmalm~\cite{H94} and Shimorin~\cite{Shimorin}.
Recently, biharmonic Bergman space potential theory 
has been applied by Hedenmalm, Shimorin and others 
in the study of Hele-Shaw flows and related problems 
of differential geometric nature (see~\cite{HO,HP,HS}). 
In recent work the author has indicated how related methods 
can be used to develop a generalized systems theory for 
weighted Bergman space norms 
(see~\cite{O06-2,Osystem,Oexpansive}).

\section{Boundary behavior of building block functions}\label{boundarybehaviorbbfcns}

\setcounter{equation}{0}
\setcounter{thm}{0}
\setcounter{prop}{0}
\setcounter{lemma}{0}
\setcounter{cor}{0}
\setcounter{remark}{0}

Functions of the form $u(z)=(1-\vert z\vert^2)^{\alpha}/\vert1-z\vert^{2\beta}$ 
appear in the formulas for Poisson kernels. 
In this section we shall discuss the distributional boundary value and 
normal derivative of such a function.
We use the standard notation 
$$
\hat{f}(k)=\frac{1}{2\pi}\int_\T f(e^{i\theta})e^{-ik\theta}d\theta,
\quad k\in\Z,
$$
for the Fourier coefficients of an integrable function $f\in L^1(\T)$ 
and similarly for distributions $f\in\De'(\T)$.

Let us consider first the standard Poisson kernel 
$$
P(z)=\frac{1-\vert z\vert^2}{\vert1-z\vert^{2}},
\quad z\in\D,
$$
for the unit disc. It is well-known that $\lim_{r\to1}P_r=\delta_1$ 
in the weak$^*$ topology of measures (distributions of order $0$). 
A straightforward computation shows that
$\partial_nP=-\kappa$ in $\De'(\T)$, where $\kappa$ is the distribution
$$
\kappa=\sum_{k=-\infty}^\infty\vert k\vert e^{ik\theta}
\quad\text{in}\ \De'(\T).
$$ 
The distribution $\kappa$ can also be described as the (tangential) 
distributional derivative of the principal value 
distribution of $\cot(\theta/2)$:
$$
\kappa=\frac{d}{d\theta}\pv\cot(\theta/2)\quad\text{in}\ \De'(\T).
$$
Recall also the so-called conjugate function distribution $\tilde{f}$ 
for $f\in\De'(\T)$ which is defined by 
$$
\tilde{f}=(\pv\cot(\theta/2))*f
=-i\sum_{k=-\infty}^\infty\sgn(k)\hat{f}(k)e^{ik\theta}
\quad\text{in}\ \De'(\T),  
$$
where $\sgn(k)=k/\vert k\vert$ for $k\neq0$ and $\sgn(0)=0$
(see~\cite{Katznelson,Zygmund}). 
In particular, the distribution $\kappa$ is of order $2$.
We mention also that 
$$
\kappa=\lim_{r\to1}2\Re(K_r)\quad\text{in}\ \De'(\T), 
$$ 
where $K(z)=z/(1-z)^2$ is the well-known Koebe function 
from conformal mapping theory.

We now consider the case $\beta\geq2$.
We denote by $\langle\cdot,\cdot\rangle$ the standard distributional pairing.

\begin{thm}\label{asymptoticbdrybehavior}
Let $\beta\geq2$ be an integer, and let
$$
u(z)=\frac{(1-\vert z\vert^2)^{2\beta-1}}{\vert1-z\vert^{2\beta}},
\quad z\in\D.
$$
Then for $\varphi\in C^2(\T)$ we have the asymptotic expansion
$$
\langle u_r,\varphi\rangle=a\varphi(1)+b(1-r)\varphi(1)+O((1-r)^2)
$$ 
as $r\to1$ for some real constants $a$ and $b$. 
In particular, the distributional boundary value 
and inward normal derivative of $u$ 
are given by 
$$
f_0=\lim_{r\to1}u_r=a\delta_1\quad\text{and}\quad
f_1=\partial_n u=
\lim_{r\to1}(u_r-f_0)/(1-r)=b\delta_1\quad\text{in}\ \De'(\T).
$$ 
\end{thm}

\begin{proof}
We first observe that
$$
\langle u_r,\varphi\rangle=\frac{1}{2\pi}\int_\T 
\frac{(1-r^2)^{2\beta-1}}{\lvert1-re^{i\theta}\rvert^{2\beta}}
\varphi(e^{i\theta})d\theta
=\frac{1}{2\pi}\int_\T 
\frac{(1-r^2)^{2\beta-1}}{\lvert1-re^{i\theta}\rvert^{2\beta}}
\varphi(e^{-i\theta})d\theta.
$$
We now take an average to conclude that 
$$
\langle u_r,\varphi\rangle-\hat{u}_r(0)\varphi(1)=
\frac{1}{2\pi}\int_\T 
\frac{(1-r^2)^{2\beta-1}}{\lvert1-re^{i\theta}\rvert^{2\beta}}
\Big(\frac{\varphi(e^{i\theta})+\varphi(e^{-i\theta})}{2}-\varphi(1)\Big)
d\theta,
$$
where $\hat{u}_r(0)$ is the $0$-th Fourier coefficient of the function $u_r$. 
We next rewrite this last integral as 
$$
(1-r^2)^2\frac{1}{2\pi}\int_\T 
\frac{(1-r^2)^{2\beta-3}}{\lvert1-re^{i\theta}\rvert^{2\beta-2}}
\frac{1}{\lvert1-re^{i\theta}\rvert^{2}}
\Big(\frac{\varphi(e^{i\theta})+\varphi(e^{-i\theta})}{2}-\varphi(1)\Big)
d\theta.
$$
The function within parenthesis in the integrand is 
$O(\theta^2)$ as $\theta\to0$ since $\varphi\in C^2(\T)$, 
and so the product of the last two factors is bounded in absolute value 
uniformly in $e^{i\theta}\in\T$ and $0\leq r<1$. 
We also notice that the integral means of the function    
$(1-r^2)^{2\beta-3}/\lvert1-re^{i\theta}\rvert^{2\beta-2}$ are uniformly bounded 
for $0\leq r<1$ (see Proposition~\ref{integralmeanexpansion} below).  
This proves the estimate 
$$
\vert\langle u_r,\varphi\rangle-\hat{u}_r(0)\varphi(1)\vert\leq C(1-r)^2,
\quad 0\leq r<1,
$$
where $C$ is an absolute constant.
A Taylor expansion argument now  
yields the conclusion of the theorem.
\end{proof}

We shall next compute the numbers $a$ and $b$ 
in Theorem~\ref{asymptoticbdrybehavior}.

\begin{prop}\label{integralmeanexpansion}
Let $\beta\geq2$ be an integer, 
and let $u$, $a$ and $b$ be as in Theorem~\ref{asymptoticbdrybehavior}. 
Then
$$
\frac{1}{2\pi}\int_\T u(re^{i\theta})d\theta
=a+b(1-r)+O((1-r)^2)
$$
as $r\to1$,
and the numbers $a$ and $b$ are given by the sums 
$$ 
a=\sum_{k=0}^{2\beta-2}
\sum_{j=0}^{k}(-1)^j\binom{2\beta-1}{j}\binom{k-j+\beta-1}{k-j}^2
$$
and 
$$
b=-\sum_{k=1}^{2\beta-2}
2k\sum_{j=0}^{k}(-1)^j\binom{2\beta-1}{j}\binom{k-j+\beta-1}{k-j}^2.
$$ 
\end{prop}

\begin{proof}
Putting $\varphi=1$ in Theorem~\ref{asymptoticbdrybehavior} 
yields the asymptotic expansion for the integral means of $u$.
It remains to prove the last equalities for $a$ and $b$.
Recall the standard power series expansion 
$$
\frac{1}{(1-z)^\beta}=\sum_{k\geq0}\binom{k+\beta-1}{k}z^k,
\quad z\in\D.
$$ 
By the Parseval formula we have that 
\begin{align*}
\frac{1}{2\pi}\int_\T 
\frac{(1-r^2)^{2\beta-1}}{\vert1-re^{i\theta}\vert^{2\beta}}d\theta
&=(1-r^2)^{2\beta-1}\sum_{k\geq0}\binom{k+\beta-1}{k}^2r^{2k}\\
&=\sum_{k\geq0}\Big(\sum_{j=0}^{\min(2\beta-1,k)}
(-1)^j\binom{2\beta-1}{j}\binom{k-j+\beta-1}{k-j}^2\Big)r^{2k}
\end{align*}
for $0\leq r<1$. Notice that the binomial coefficient
$$
\binom{k+\beta-1}{k}=\frac{1}{(\beta-1)!}\prod_{j=1}^{\beta-1}(k+j)
$$  
is a polynomial in $k$ of degree $\beta-1$. 
By a backward difference argument we have 
$$
\sum_{j=0}^{\min(2\beta-1,k)}
(-1)^j\binom{2\beta-1}{j}\binom{k-j+\beta-1}{k-j}^2=0
$$
for $k\geq2\beta-1$, showing that 
\begin{equation}\label{integralmeansformula}
\frac{1}{2\pi}\int_\T u(re^{i\theta})d\theta
=\sum_{k=0}^{2\beta-2}\Big(\sum_{j=0}^{k}
(-1)^j\binom{2\beta-1}{j}\binom{k-j+\beta-1}{k-j}^2\Big)r^{2k}
\end{equation}
for $0\leq r<1$. 
In particular, this last sum is a polynomial in $r^2$  
of degree at most $2\beta-2$. 
By a Taylor expansion argument we arrive at the formulas for  
$a$ and $b$ in the proposition. 
\end{proof}

\begin{remark}
We remark that the integral means of $u$ are given by 
formula \eqref{integralmeansformula} for $0\leq r<1$.
\end{remark}

\begin{remark}\label{asymptoticssums}
Denote by $a_\beta$ and $b_\beta$ the numbers $a$ and $b$ 
in Proposition~\ref{integralmeanexpansion} for a given integer $\beta\geq2$. 
A computation shows that the first few of these numbers are given by  
$$
(a_2,b_2)=(2,-2),\quad (a_3,b_3)=(6,-12)\quad\text{and}\quad(a_4,b_4)=(20,-60).
$$
\end{remark}

\begin{remark}
It is easy to see that the set of functions of the form 
$u(z)=(1-\vert z\vert^2)^{\alpha}/\vert1-z\vert^{2\beta}$, 
where $\alpha,\beta\geq1$, 
are linearly independent. Indeed, if 
$$
\sum_{\alpha,\beta\geq1}c_{\alpha\beta}
\frac{(1-\vert z\vert^2)^{\alpha}}{\vert1-z\vert^{2\beta}}=0,\quad z\in\D,
$$ 
then 
$\sum_{\alpha,\beta\geq1}c_{\alpha\beta}s^\alpha t^\beta=0$ 
for every point $(s,t)\in\R^2$ 
of the form $s=1-\vert z\vert^2$, $t=\vert1-z\vert^{-2}$, where $z\in\D$, 
which implies that $c_{\alpha\beta}=0$ for all $\alpha,\beta\geq1$.
\end{remark}

\section{Decoupling of Laplacians}\label{decouplingLaplacians}

\setcounter{equation}{0}
\setcounter{thm}{0}
\setcounter{prop}{0}
\setcounter{lemma}{0}
\setcounter{cor}{0}
\setcounter{remark}{0}

We shall describe in this section a method of decoupling of Laplacians 
that we have used to compute $F_\gamma$ and $H_\gamma$'s.
For notational reasons
we introduce the ordinary differential operators
$$
P_\beta=P_\beta(x,\frac{d}{dx})=(1-\beta)\frac{d}{dx}+x\frac{d^2}{dx^2}
$$
and
$$
Q_\beta=Q_\beta(x,\frac{d}{dx})=\beta\Big(\beta+(1-x)\frac{d}{dx}\Big)
$$
for $\beta\geq1$. 

\begin{lemma}\label{Lplcomp}
Let $u$ be a function of the form 
$$
u(z)=\frac{f(\vert z\vert^2)}{\vert1-z\vert^{2\beta}},\quad z\in\D,
$$
where $f$ is $C^2$-smooth on the interval $[0,1)$ and $\beta\geq1$.
Then
$$
\Delta u(z)=\frac{(P_\beta f)(\vert z\vert^2)}{\vert1-z\vert^{2\beta}}+
\frac{(Q_\beta f)(\vert z\vert^2)}{\vert1-z\vert^{2(\beta+1)}},
\quad z\in\D.
$$
\end{lemma}

\begin{proof}
For the sake of completeness we include some details of proof.
Differentiating we find that
\begin{align*}
\frac{\partial u}{\partial\bar z}(z)&=
f'(\vert z\vert^2)z\frac{1}{\vert1-z\vert^{2\beta}}
+f(\vert z\vert^2)(-\beta)(\vert1-z\vert^2)^{-\beta-1}
\frac{\partial}{\partial\bar z}(\vert1-z\vert^2)\\
&=f'(\vert z\vert^2)z\frac{1}{\vert1-z\vert^{2\beta}}+
f(\vert z\vert^2)(-\beta)\frac{1}{\vert1-z\vert^{2(\beta+1)}}(z-1).
\end{align*}
By another differentiation we have that
\begin{align*}
\frac{\partial^2 u}{\partial z\partial\bar z}(z)=&
(f'(\lvert z\rvert^2)+\lvert z\rvert^2f''(\lvert z\rvert^2))
\frac{1}{\lvert1-z\rvert^{2\beta}}\\
&+\beta(\beta f(\lvert z\rvert^2)-2\lvert z\rvert^2f'(\lvert z\rvert^2))
\frac{1}{\lvert1-z\rvert^{2(\beta+1)}}\\
&+\beta f'(\lvert z\rvert^2)2\Re(z)\frac{1}{\lvert1-z\rvert^{2(\beta+1)}},
\end{align*}
where $\Re(z)$ denotes the real part of the complex number $z$.
We now use the formula $2\Re(z)=1+\vert z\vert^2-\vert1-z\vert^2$ 
to rewrite this last sum as
\begin{align*}
\Delta u(z)=&
\Big((1-\beta)f'(\vert z\vert^2))+\vert z\vert^2f''(\vert z\vert^2)\Big)
\frac{1}{\vert1-z\vert^{2\beta}}\\ &+
\beta\Big(\beta f(\vert z\vert^2)+(1-\vert z\vert^2)f'(\vert z\vert^2)\Big)
\frac{1}{\vert1-z\vert^{2(\beta+1)}},\quad z\in\D.
\end{align*}
This completes the proof of the lemma.
\end{proof}

Let $w:\D\to(0,\infty)$ be a smooth radial weight function and write
\begin{equation}\label{tildeweight} 
w(z)=\tilde w(\lvert z\rvert^2),\quad z\in\D.
\end{equation}
Our construction of nontrivial $w$-biharmonic functions in 
Section~\ref{F2H2computation} use systems 
of functions $\{f_\beta\}_{\beta\geq1}$ in, say, $C^4[0,1)$ 
satisfying the system of ordinary differential equations
\begin{align}
&P_1\tilde w^{-1} P_1f_1=0,\label{P1P1eq}\\
&(P_2\tilde w^{-1}Q_1+Q_1\tilde w^{-1}P_1)f_1+P_2\tilde w^{-1}P_2f_2=0
\label{f1f2eq}
\end{align}
and 
\begin{align}
\label{fbreceq}
Q_{\beta+1}\tilde w^{-1}Q_\beta f_\beta +
(P_{\beta+2}\tilde w^{-1}Q_{\beta+1}&+Q_{\beta+1}\tilde w^{-1}P_{\beta+1})f_{\beta+1}\\
&+P_{\beta+2}\tilde w^{-1}P_{\beta+2}f_{\beta+2}=0\notag
\end{align}
for $\beta\geq1$.

\begin{prop}\label{constructionbiharmonicfcn}
Let  $\{f_\beta\}_{\beta\geq1}$ be a system of smooth functions 
on $[0,1)$ satisfying 
equations \eqref{P1P1eq}, \eqref{f1f2eq} and \eqref{fbreceq} for $\beta\geq1$, 
and assume that the sum  
\begin{equation}\label{ufcn}
u(z)=\sum_{\beta\geq1}\frac{f_\beta(\vert z\vert^2)}{\vert1-z\vert^{2\beta}},
\quad z\in\D,
\end{equation}
is suitably convergent.
Then $\Delta w^{-1}\Delta u=0$ in $\D$.
\end{prop}

\begin{proof}
Let $u$ be a function of the form \eqref{ufcn} and assume that 
the sum in \eqref{ufcn} is convergent in the sense of distributions in $\D$. 
Computing the Laplacian of $u$ using Lemma~\ref{Lplcomp} we have that
\begin{align*}
\Delta u(z)&=\sum_{\beta\geq1}\Big(
\frac{P_\beta f_\beta(\vert z\vert^2)}{\vert1-z\vert^{2\beta}}+  
\frac{Q_\beta f_\beta(\vert z\vert^2)}{\vert1-z\vert^{2(\beta+1)}}
\Big)\\
&=\frac{P_1f_1(\vert z\vert^2)}{\vert1-z\vert^2}+
\sum_{\beta\geq2}\frac{P_\beta f_\beta(\vert z\vert^2)
+Q_{\beta-1}f_{\beta-1}(\vert z\vert^2)}{\vert1-z\vert^{2\beta}}.
\end{align*}
Applying another Laplacian and rearranging terms we arrive at the formula
\begin{align*}
\Delta w^{-1}\Delta u(z)&=
\frac{P_1\tilde w^{-1}P_1f_1(\vert z\vert^2)}{\vert1-z\vert^2} \\
&\quad+
\frac{P_2\tilde w^{-1}P_2f_2(\vert z\vert^2) 
+(P_2\tilde w^{-1}Q_1+
Q_1\tilde w^{-1}P_1)f_1(\vert z\vert^2)}{\vert1-z\vert^4}\\
&+
\sum_{\beta\geq3}\Big(
P_\beta\tilde w^{-1}P_\beta f_\beta(\vert z\vert^2)
+(P_{\beta}\tilde w^{-1}Q_{\beta-1}+
Q_{\beta-1}\tilde w^{-1}P_{\beta-1})f_{\beta-1}(\vert z\vert^2)\\
&\quad+
Q_{\beta-1}\tilde w^{-1}Q_{\beta-2}f_{\beta-2}(\vert z\vert^2)\Big)
\frac{1}{\vert1-z\vert^{2\beta}}.
\end{align*}
We now conclude that $\Delta w^{-1}\Delta u=0$ in $\D$ if 
$\{f_\beta\}_{\beta\geq1}$ 
satisfies the system \eqref{P1P1eq}-\eqref{fbreceq}
of differential equations. 
\end{proof}

Equations \eqref{P1P1eq}-\eqref{fbreceq} also have 
a certain structure of a Jacobi matrix.
We introduce the operations $P$ and $Q$ operating  
on sequences $f=\{f_\beta\}_{\beta\geq1}$ of smooth functions by 
$$
Pf=\{P_\beta f_\beta\}_{\beta\geq1}
$$
and
$$
Qf=\{Q_{\beta-1}f_{\beta-1}\}_{\beta\geq1},
$$
where $Q_0f_{0}=0$. In this terminology the Laplacian corresponds to the operation 
$$
(P+Q)f=\{P_\beta f_\beta+Q_{\beta-1}f_{\beta-1}\}_{\beta\geq1}
$$
(see Lemma~\ref{Lplcomp}), 
and equations \eqref{P1P1eq}-\eqref{fbreceq} can be written
$$
(P+Q)\tilde w^{-1}(P+Q)f=0.
$$
Notice that the operation $P+Q$ takes the form 
$$
(P+Q)f=\left[\begin{matrix} P_1 & & \\ Q_1 & P_2 & \\ &\ddots & \ddots \end{matrix}\right]
\left[\begin{matrix} f_1\\ f_2\\  \vdots \end{matrix}\right]
$$  
using standard block matrix notation.

We write 
$$
\tilde w_\gamma(x)=(1-x)^\gamma,\quad x\in[0,1),
$$ 
where $\gamma>-1$, 
in accordance with \eqref{tildeweight}.

For easy reference we record the following lemma.

\begin{lemma}\label{diffopmonomial}
Let $k$ be a non-negative integer.  
Then
\begin{align*}
&Q_{\beta+1}\tilde w_\gamma^{-1}Q_\beta(1-x)^k
=\beta(\beta+1)(\beta-k)(\beta+\gamma+1-k)(1-x)^{k-\gamma},\\
&(P_{\beta+1}\tilde w_\gamma^{-1}Q_\beta+Q_\beta\tilde w_\gamma^{-1}P_\beta)(1-x)^k
=\beta(\beta+\gamma+1-k)(\beta-k)(2k-\gamma)(1-x)^{k-\gamma-1}\\
&\quad +
\beta(k(k-1)(\beta+\gamma+2-k)+(\beta-k)(k-\gamma)(k-\gamma-1))(1-x)^{k-\gamma-2}
\end{align*}
and
\begin{align*}
P_{\beta} \tilde w_\gamma^{-1}P_\beta(1-x)^k
&=k(\beta-k)(k-\gamma-1)(\beta+\gamma+1-k)(1-x)^{k-\gamma-2}\\
&\quad+k(k-\gamma-2)\Big((\beta-k)(k-\gamma-1)\\
&\qquad
+(\beta-1)(k-1)-(k-1)(k-\gamma-3)\Big)(1-x)^{k-\gamma-3}\\
&\quad+k(k-1)(k-\gamma-2)(k-\gamma-3)(1-x)^{k-\gamma-4}.
\end{align*}
\end{lemma}

\begin{proof}
Straightforward computation. We omit the details.
\end{proof}

We remark that the function 
$f(x)=(1-x)^{\beta+\gamma+1}$ satisfies 
the equation 
$$
Q_{\beta+1}\tilde w_\gamma^{-1}Q_\beta f=0.
$$
Terms of this type appear in formulas for 
biharmonic Poisson kernels $F_\gamma$ and $H_\gamma$.  

We shall next comment on how we choose leading and first terms 
in our calculation of nontrivial solutions for
\eqref{P1P1eq}-\eqref{fbreceq} with $\tilde w=\tilde w_\gamma$.
The function $f_1$ is chosen of the form 
$$
f_1(x)=c_1(1-x)^{\gamma+2},
$$
which ensures that \eqref{P1P1eq} is satisfied. 
Notice that with this choice of $f_1$ we have 
$Q_2\tilde w_\gamma^{-1}Q_1f_1=0$,
and that equation \eqref{fbreceq} for $\beta=1$ simplifies to 
\begin{equation}\label{fbreceqsimpl}
(P_{3}\tilde w_\gamma^{-1}Q_{2}+Q_{2}\tilde w_\gamma^{-1}P_{2})f_{2}
+P_{3}\tilde w_\gamma^{-1}P_{3}f_{3}=0.
\end{equation}
This choice of $f_1$ is motivated as follows.
The function $f_1$ must satisfy \eqref{P1P1eq}, that is,
$$
x\frac{d^2}{dx^2}\tilde w_\gamma^{-1}x\frac{d^2}{dx^2}f_1=0.
$$
We further want this function $f_1$ to be smooth in $[0,1)$ 
which leaves us to the possibility that $f_1$ has the form 
$$
f_1(x)=c_1(1-x)^{\gamma+2}+ax+b,
$$
where $c_1$, $a$ and $b$ are constants. 
We also want the function $f_1(\vert z\vert^2)/\vert1-z\vert^2$ 
to have distributional boundary value and normal derivative 
equal to constant multiples of $\delta_1$ which forces $a=b=0$
(see Section~\ref{boundarybehaviorbbfcns}). 
This leaves us to the only possibility that 
$f_1(x)=c_1(1-x)^{\gamma+2}$.

Let us now discuss how we choose leading terms. 
Let $\beta_0\geq1$ be such that $f_{\beta_0}$ is not identically zero 
and $f_{\beta}=0$ for $\beta>\beta_0$. 
Then by \eqref{fbreceq} for $\beta=\beta_0$ we must have   
$$
Q_{\beta_0+1}\tilde w_\gamma^{-1}Q_{\beta_0}f_{\beta_0}=0,
$$
which gives that $f_{\beta_0}$ must be of the form 
$$
f_{\beta_0}(x)=c(1-x)^{\beta_0+\gamma+1}+a(1-x)^{\beta_0}
$$
for some constants $c$ and $a$.
Suppose further that we want 
the leading term 
$f_{\beta_0}(\vert z\vert^2)/\vert1-z\vert^{2\beta_0}$ 
to have boundary value equal to a nonzero constant multiple of $\delta_1$
and normal derivative equal to a constant multiple of $\delta_1$.
Then by results from Section~\ref{boundarybehaviorbbfcns} 
we must have $a=0$ and 
$\beta_0+\gamma+1=2\beta_0-1$, that is, 
we set  $\beta_0=\gamma+2$ and $f_{\gamma+2}(x)=c(1-x)^{2\gamma+3}$. 
 
Similarly, if we want the term 
$f_{\beta_0}(\vert z\vert^2)/\vert1-z\vert^{2\beta_0}$
to have vanishing boundary value and normal derivative 
equal to a nonzero constant multiple of $\delta_1$, 
then we choose $a=0$ and $\beta_0$ such that $2\beta_0=\beta_0+\gamma+1$, 
that is, we set $\beta_0=\gamma+1$ and $f_{\gamma+1}(x)=c(1-x)^{2\gamma+2}$.

\section{Computation of $F_2$ and $H_2$}\label{F2H2computation}

\setcounter{equation}{0}
\setcounter{thm}{0}
\setcounter{prop}{0}
\setcounter{lemma}{0}
\setcounter{cor}{0}
\setcounter{remark}{0}

In this section we shall derive formulas for 
the Poisson kernels $F_2$ and $H_2$. 
The construction proceeds from basic principles.
Recall the notation
$$
\tilde w_2(x)=(1-x)^2,\quad x\in[0,1),
$$
introduced in Section~\ref{decouplingLaplacians}.

We first construct $H_2$ up to a constant multiple.

\begin{prop}\label{H2typeconstruct}
The function 
$$
u(z)=3\frac{(1-\vert z\vert^2)^4}{\vert1-z\vert^2}
+3\frac{(1-\vert z\vert^2)^5}{\vert1-z\vert^4}
-\frac{3}{2}\frac{(1-\vert z\vert^2)^4}{\vert1-z\vert^4}
+\frac{(1-\vert z\vert^2)^6}{\vert1-z\vert^6},
\quad z\in\D,
$$
is $w_2$-biharmonic in $\D$.
\end{prop}

\begin{proof}
We construct a solution $\{f_\beta\}_{\beta\geq1}$ of 
\eqref{P1P1eq}-\eqref{fbreceq} for $w=w_2$. 
Set $f_3(x)=(1-x)^{6}$ and 
$f_\beta(x)=0$ for $\beta\geq4$. 
Notice that $Q_4\tilde w_2^{-1}Q_3f_3=0$ (see Lemma~\ref{diffopmonomial}) 
so that \eqref{fbreceq} holds for $\beta\geq3$.

We next search for a function $f_2$ satisfying \eqref{fbreceq} for $\beta=2$. 
By Lemma~\ref{diffopmonomial} we have that 
$$
(P_4\tilde w_2^{-1}Q_3+Q_3\tilde w_2^{-1}P_3)f_3=-18(1-x)^2.
$$
We now set $f_2(x)=-\frac{3}{2}(1-x)^4+c_2(1-x)^5$, 
where $c_2$ is a constant to be determined. 
By Lemma~\ref{diffopmonomial} we have 
$$
Q_3\tilde w_2^{-1}Q_2f_2=18(1-x)^2,
$$
showing that \eqref{fbreceq} holds for $\beta=2$.

We set $f_1(x)=c_1(1-x)^{4}$, where $c_1$ is a constant to be determined. 
A computation using Lemma~\ref{diffopmonomial} 
gives $P_1\tilde w_2^{-1}P_1f_1=0$, 
showing that \eqref{P1P1eq} is satisfied.
Notice also that $Q_2\tilde w_2^{-1}Q_1f_1=0$.

We proceed to determine the constants $c_1$ and $c_2$ in such a way that 
\eqref{f1f2eq} and \eqref{fbreceq} for $\beta=1$ are satisfied.
We consider first \eqref{f1f2eq}. 
A computation using  Lemma~\ref{diffopmonomial} gives 
$$
(P_2\tilde w_2^{-1}Q_1+Q_1\tilde w_2^{-1}P_1)f_1=6c_1,
$$
and similarly that
$$
P_2\tilde w_2^{-1}P_2f_2=12-10c_2.
$$
Now
$$
(P_2\tilde w_2^{-1}Q_1+Q_1\tilde w_2^{-1}P_1)f_1
+P_2\tilde w_2^{-1}P_2f_2=12-10c_2+6c_1,
$$  
showing that \eqref{f1f2eq} is satisfied if and only if $12-10c_2+6c_1=0$.

Let us now consider \eqref{fbreceq} for $\beta=1$ 
which simplifies to \eqref{fbreceqsimpl}. 
A computation using Lemma~\ref{diffopmonomial} 
gives 
$$
(P_3\tilde w_2^{-1}Q_2+Q_2\tilde w_2^{-1}P_2)f_2=-60+(36+4c_2)(1-x),
$$
and similarly that 
$$
P_3\tilde w_2^{-1}P_3f_3=60-48(1-x).
$$
Now 
$$
(P_3\tilde w_2^{-1}Q_2+Q_2\tilde w_2^{-1}P_2)f_2
+P_3\tilde w_2^{-1}P_3f_3
=(-12+4c_2)(1-x),
$$
showing that \eqref{fbreceqsimpl} is fulfilled if and only if $c_2=3$.
Setting $c_1=c_2=3$ we obtain a solution $\{f_\beta\}_{\beta=1}^3$ of 
equations \eqref{P1P1eq}-\eqref{fbreceq} for $w=w_2$. 
The conclusion of the proposition 
now follows by Proposition~\ref{constructionbiharmonicfcn}. 
\end{proof}

We now compute $H_2$.

\begin{thm}\label{H2formula}
The function $H_2$ is given by the formula
$$
H_2(z)=\frac{1}{2}\frac{(1-\vert z\vert^2)^4}{\vert1-z\vert^2}
+\frac{1}{2}\frac{(1-\vert z\vert^2)^5}{\vert1-z\vert^4}
-\frac{1}{4}\frac{(1-\vert z\vert^2)^4}{\vert1-z\vert^4}
+\frac{1}{6}\frac{(1-\vert z\vert^2)^6}{\vert1-z\vert^6},
\quad z\in\D.
$$
\end{thm}

\begin{proof}
Denote by $u$ the function in Proposition~\ref{H2typeconstruct} 
which we know is $w_2$-biharmonic. 
We compute the asymptotics of the integral means of $u$. 
By Proposition~\ref{integralmeanexpansion} and 
Remark~\ref{asymptoticssums} we have that
$$
\frac{1}{2\pi}\int_\T u(re^{i\theta})d\theta
=-\frac{3}{2}(1-r^2)2+(1-r^2)6+O((1-r)^2)=6(1-r)+O((1-r)^2)
$$
as $r\to1$. By Theorem~\ref{asymptoticbdrybehavior} 
we have that $u$ solves the Dirichlet problem (\ref{Dirichletproblem}) 
for $w=w_2$ in the distributional sense with $f_0=0$ and $f_1=6\delta_1$.
By uniqueness of solutions of (\ref{Dirichletproblem}) 
we conclude that $u=6H_2$ in $\D$ 
(see~\cite[Theorem~2.1]{Olofsson05}). 
Solving for $H_2$ gives the conclusion of the theorem.
\end{proof}

We next construct a nontrivial $w_2$-biharmonic function with 
boundary value equal to a constant multiple of $\delta_1$.

\begin{prop}\label{F2typeconstruct}
The function 
\begin{align*}
u(z)=&-8\frac{(1-\vert z\vert^2)^4}{\vert1-z\vert^2}
+\frac{3}{2}\frac{(1-\vert z\vert^2)^4}{\vert1-z\vert^4}
-6\frac{(1-\vert z\vert^2)^5}{\vert1-z\vert^4}\\
&-3\frac{(1-\vert z\vert^2)^5}{\vert1-z\vert^6}
+\frac{(1-\vert z\vert^2)^7}{\vert1-z\vert^8},
\quad z\in\D,
\end{align*}
is $w_2$-biharmonic in $\D$.
\end{prop}

\begin{proof}
We construct a solution $\{f_\beta\}_{\beta\geq1}$ of 
\eqref{P1P1eq}-\eqref{fbreceq} for $w=w_2$.
Put $f_4(x)=(1-x)^{7}$ and $f_\beta=0$ for $\beta>4$. 
By  Lemma~\ref{diffopmonomial}  
we have $Q_5\tilde w_2^{-1}Q_4f_4=0$, 
so that \eqref{fbreceq} holds for $\beta\geq4$.

We proceed to choose $f_3$ such that \eqref{fbreceq} holds for $\beta=3$. 
By Lemma~\ref{diffopmonomial} we have that
$$
(P_5\tilde w_2^{-1}Q_4+Q_4\tilde w_2^{-1}P_4)f_4=-72(1-x)^3.
$$
We set $f_3(x)=-3(1-x)^5$. Then $Q_4\tilde w_2^{-1}Q_3f_3=72(1-x)^3$, 
showing that \eqref{fbreceq} holds for $\beta=3$.

We proceed to choose $f_2$ such that \eqref{fbreceq} holds for $\beta=2$. 
A computation using Lemma~\ref{diffopmonomial} 
gives 
$$
(P_4\tilde w_2^{-1}Q_3+Q_3\tilde w_2^{-1}P_3)f_3=144(1-x)^2-252(1-x),
$$ 
and similarly that  
$$
P_4\tilde w_2^{-1}P_4f_4=-126(1-x)^2+252(1-x).
$$
Now 
$$
(P_4\tilde w_2^{-1}Q_3+Q_3\tilde w_2^{-1}P_3)f_3+P_4\tilde w_2^{-1}P_4f_4=18(1-x)^2.
$$
We now set $f_2(x)=\frac{3}{2}(1-x)^4+c_2(1-x)^5$, 
where $c_2$ is a constant to be determined. 
By Lemma~\ref{diffopmonomial} 
we have 
$Q_3\tilde w_2^{-1}Q_2f_2=-18(1-x)^2$, 
showing that \eqref{fbreceq} holds for $\beta=2$.

We set $f_1(x)=c_1(1-x)^4$, 
where $c_1$ is a constant to be determined. 
Then $P_1\tilde w_2^{-1}P_1f_1=0$ by Lemma~\ref{diffopmonomial}, 
showing that \eqref{P1P1eq} is satisfied. 
Notice also that $Q_2\tilde w_2^{-1}Q_1f_1=0$.

It remains to choose the constants $c_1$ and $c_2$ such that 
\eqref{f1f2eq} and \eqref{fbreceq} for $\beta=1$ are fulfilled.
We consider first \eqref{f1f2eq}.
By Lemma~\ref{diffopmonomial} 
we have
$$
(P_2\tilde w_2^{-1}Q_1+Q_1\tilde w_2^{-1}P_1)f_1=6c_1,
$$
and also that
$$
P_2\tilde w_2^{-1}P_2f_2=-12-10c_2.
$$
Now 
$$
(P_2\tilde w_2^{-1}Q_1+Q_1\tilde w_2^{-1}P_1)f_1+P_2\tilde w_2^{-1}P_2f_2=-12+6c_1-10c_2,
$$
showing that \eqref{f1f2eq} is satisfied if and only if $-12+6c_1-10c_2=0$.

We next consider \eqref{fbreceq} for $\beta=1$ 
which simplifies to \eqref{fbreceqsimpl}. 
A computation using Lemma~\ref{diffopmonomial} 
shows 
$$
(P_3\tilde w_2^{-1}Q_2+Q_2\tilde w_2^{-1}P_2)f_2=60+(4c_2-36)(1-x),
$$
and similarly that 
$$
P_3\tilde w_2^{-1}P_3f_3=-60+60(1-x).
$$
Now 
$$
(P_3\tilde w_2^{-1}Q_2+Q_2\tilde w_2^{-1}P_2)f_2+P_3\tilde w_2^{-1}P_3f_3=(24+4c_2)(1-x),
$$
showing that \eqref{fbreceqsimpl} is fulfilled if and only if
$c_2=-6$. We conclude that both \eqref{f1f2eq} and 
\eqref{fbreceqsimpl} are satisfied if and only if $c_1=-8$ and $c_2=-6$.
Setting $c_1=-8$ and $c_2=-6$ we obtain a solution 
$\{f_\beta\}_{\beta=1}^4$ of \eqref{P1P1eq}-\eqref{fbreceq} for $w=w_2$. 
The conclusion of the proposition now follows by 
Proposition~\ref{constructionbiharmonicfcn}.
\end{proof}

We now compute $F_2$.

\begin{thm}\label{F2formula}
The function $F_2$ is given by the formula
\begin{align*}
F_2(z)=&\frac{1}{2}\frac{(1-\vert z\vert^2)^4}{\vert1-z\vert^2}
+\frac{3}{2}\frac{(1-\vert z\vert^2)^5}{\vert1-z\vert^4}
-\frac{3}{2}\frac{(1-\vert z\vert^2)^4}{\vert1-z\vert^4}\\
&+\frac{3}{2}\frac{(1-\vert z\vert^2)^6}{\vert1-z\vert^6}   
-\frac{3}{2}\frac{(1-\vert z\vert^2)^5}{\vert1-z\vert^6}
+\frac{1}{2}\frac{(1-\vert z\vert^2)^7}{\vert1-z\vert^8},
\quad z\in\D.
\end{align*}
\end{thm}

\begin{proof}
Denote by $u$ be the function in Proposition~\ref{F2typeconstruct} 
which we know is $w_2$-biharmonic. 
We compute the asymptotics of the integral means of $u$. 
By Proposition~\ref{integralmeanexpansion} and 
Remark~\ref{asymptoticssums} we have that
\begin{align*}
\frac{1}{2\pi}\int_\T u(re^{i\theta})d\theta
&=\frac{3}{2}(1-r^2)2-3(6-12(1-r))+20-60(1-r)+O((1-r)^2)\\
&=2-18(1-r)+O((1-r)^2)
\end{align*}
as $r\to1$. 
By Theorem~\ref{asymptoticbdrybehavior} 
we have that $u$ solves the Dirichlet problem (\ref{Dirichletproblem}) 
for $w=w_2$ in the distributional sense with $f_0=2\delta_1$ and $f_1=-18\delta_1$.
By uniqueness of solutions of (\ref{Dirichletproblem}) 
we conclude that $u=2F_2-18H_2$ in $\D$ 
(see~\cite[Theorem~2.1]{Olofsson05}). 
Solving for $F_2$ using the formula for $H_2$ in Theorem~\ref{H2formula} 
gives the conclusion of the theorem.
\end{proof}

\section{Poisson kernels for integer parameter weights}\label{expectedFHformulas}

\setcounter{equation}{0}
\setcounter{thm}{0}
\setcounter{prop}{0}
\setcounter{lemma}{0}
\setcounter{cor}{0}
\setcounter{remark}{0}

Computations of Poisson kernels $F_\gamma$ and $H_\gamma$ for integer parameter 
standard weights along the lines of what we did 
in Section~\ref{F2H2computation} for $\gamma=2$
suggest that these kernels have a certain explicit form. 
We state this in the form of two conjectures below. 
We use the symbol $\lfloor x\rfloor$ to denote the floor of a real number $x$, that is, 
the number $\lfloor x\rfloor$ is the largest integer less than or equal to $x$.

\begin{conj}\label{Fform}
For $\gamma$ a non-negative integer, the function $F_\gamma$ has the form 
$$
F_\gamma(z)=\sum_{\beta=1}^{\gamma+2}
\frac{f_\beta(\vert z\vert^2)}{\vert1-z\vert^{2\beta}},
\quad z\in\D,
$$ 
where the functions $\{f_\beta\}_{\beta=1}^{\gamma+2}$ are polynomials given by
$$
2f_\beta(x)=
\sum_{k=0}^{\beta+\gamma+1-\max(2\beta-1,\gamma+2)}
c_k\binom{\gamma+1-2k}{\beta-1-k}(1-x)^{\beta+\gamma+1-k}
$$
for $1\leq\beta\leq\gamma+2$,
where the numbers $\{c_k\}_{k=0}^{\lfloor(\gamma+1)/2\rfloor}$ 
are given by $c_0=1$ and 
$$
\sum_{k=0}^jc_k\binom{\gamma+1-2k}{j-k}=0
$$
for $1\leq j\leq\lfloor(\gamma+1)/2\rfloor$.
\end{conj}

We mention that the choice of $c_k$'s in Conjecture~\ref{Fform} ensures that 
$f_1(0)=f_{\gamma+2}(0)=1/2$ and 
$f_\beta(0)=0$ for $2\leq\beta\leq\gamma+1$. 

We conjecture that
the polynomial $f_\beta(x)$ is a certain linear combination of 
monomials $(1-x)^k$ with exponents $k$ in 
the range $\max(2\beta-1,\gamma+2)\leq k\leq\beta+\gamma+1$.
For $\beta=1,\gamma+2$ this range consists of a single exponent $k=\beta+\gamma+1$, 
whereas for intermediate values of $\beta$ the number of exponents grow 
to a maximum of $\lfloor(\gamma+1)/2\rfloor+1$ terms. 

To illustrate the statement of Conjecture~\ref{Fform} we comment on the case $\gamma=5$. 
The function $F_5$ has the form
$$
F_5(z)=\sum_{\beta=1}^7\frac{f_\beta(\vert z\vert^2)}{\vert1-z\vert^{2\beta}},
\quad z\in\D,
$$
where
\begin{align*}
2f_1(x)&=(1-x)^7,\\
2f_2(x)&=6(1-x)^8 -6(1-x)^7,\\
2f_3(x)&=15(1-x)^{9}-24(1-x)^8+9(1-x)^7,\\
2f_4(x)&=20(1-x)^{10}-36(1-x)^9+18(1-x)^8-2(1-x)^7,\\
2f_5(x)&=15(1-x)^{11}-24(1-x)^{10}+9(1-x)^9,\\
2f_6(x)&=6(1-x)^{12}-6(1-x)^{11},\\
2f_7(x)&=(1-x)^{13}.
\end{align*}
The columns in the above table of coefficients 
are up to a multiplicative constant of normalization 
rows in the standard table of binomial coefficients (Pascal's triangle):
the numbers $1$, $6$, $15$, $20$, $15$, $6$, $1$ constitute the sixth row 
$\binom{6}{k}$, $0\leq k\leq 6$, of binomial coefficients, 
the numbers $1$, $4$, $6$, $4$, $1$ are $\binom{4}{k}$, $0\leq k\leq 4$, 
the numbers $1$, $2$, $1$ are  $\binom{2}{k}$, $0\leq k\leq 2$,
and the number $1$ equals $\binom{0}{0}$. The constants of normalization 
$\{c_k\}_{k=0}^3$ are chosen such that 
$2f_1(0)=2f_7(0)=1$ and $f_\beta(0)=0$ for $2\leq\beta\leq6$, that is,
$c_0=1$, $c_1=-6$, $c_2=9$ and $c_3=-2$.

Using the computer algebra package Maxima we have verified 
Conjecture~\ref{Fform} for integer weight parameters $\gamma$ 
in the range $0\leq\gamma\leq 80$. 
This verification has been made in a straightforward manner: 
Fix $\gamma$ and let $F_\gamma$ and 
$\{f_\beta\}_{\beta=1}^{\gamma+2}$ be as in Conjecture~\ref{Fform}.
By differentiation we check that $\{f_\beta\}_{\beta=1}^{\gamma+2}$ 
satisfies \eqref{P1P1eq}-\eqref{fbreceq} for $w=w_\gamma$, 
showing that $F_\gamma$ 
so defined is $w_\gamma$-biharmonic in $\D$ by 
Proposition~\ref{constructionbiharmonicfcn}.
We have then computed the distributional boundary value 
and normal derivative of $F_\gamma$ using 
Theorem~\ref{asymptoticbdrybehavior} 
and Proposition~\ref{integralmeanexpansion} 
in Section~\ref{boundarybehaviorbbfcns} to check that $F_\gamma$ 
has the appropriate boundary values.

The corresponding statement for $H_\gamma$ reads as follows.
 
\begin{conj}\label{Hform}
For $\gamma$ a non-negative integer, the function $H_\gamma$ has the form 
$$
H_\gamma(z)=\sum_{\beta=1}^{\gamma+1}
\frac{h_\beta(\vert z\vert^2)}{\vert1-z\vert^{2\beta}},
\quad z\in\D,
$$ 
where the functions $\{h_\beta\}_{\beta=1}^{\gamma+1}$ are polynomials given by
$$
2\beta h_\beta(x)=
\sum_{k=0}^{\beta+\gamma+1-\max(2\beta,\gamma+2)}
c_k\binom{\gamma-2k}{\beta-1-k}(1-x)^{\beta+\gamma+1-k}
$$
for $1\leq\beta\leq\gamma+1$,
where the numbers $\{c_k\}_{k=0}^{\lfloor\gamma/2\rfloor}$ 
are given by $c_0=1$ and  
$$
\sum_{k=0}^jc_k\binom{\gamma-2k}{j-k}=1
$$
for $1\leq j\leq\lfloor\gamma/2\rfloor$.
\end{conj}

We remark that the last condition determining the $c_k$'s in Conjecture~\ref{Hform} 
can be stated that $2\beta h_\beta(0)=1$ for $1\leq\beta\leq\gamma+1$.

Notice that we conjecture that the polynomial $h_\beta(x)$
is a certain linear combination of 
monomials $(1-x)^k$ with exponents $k$ in 
the range $\max(2\beta,\gamma+2)\leq k\leq\beta+\gamma+1$.
For $\beta=1,\gamma+1$ this range consists of a single exponent $k=\beta+\gamma+1$, 
whereas for intermediate values of $\beta$ the number of exponents grow 
to a maximum of $\lfloor\gamma/2\rfloor+1$ terms. 

As an example we mention that the function $H_5$ has the form 
$$
H_5(z)=\sum_{\beta=1}^6\frac{h_\beta(\vert z\vert^2)}{\vert1-z\vert^{2\beta}},
\quad z\in\D,
$$
where
\begin{align*}
2h_1(x)&=(1-x)^7,\\
4h_2(x)&=5(1-x)^8-4(1-x)^7,\\
6h_3(x)&=10(1-x)^9-12(1-x)^8+3(1-x)^7,\\
8h_4(x)&=10(1-x)^{10}-12(1-x)^9+3(1-x)^8,\\
10h_5(x)&=5(1-x)^{11}-4(1-x)^{10},\\
12 h_6(x)&=(1-x)^{12}.
\end{align*}
The columns in the above table of coefficients are up to normalization 
rows in the standard table of binomial coefficients:
the numbers $1$, $5$, $10$, $10$, $5$, $1$ are 
$\binom{5}{k}$, $0\leq k\leq 5$, 
the numbers $1$, $3$, $3$, $1$ are $\binom{3}{k}$, $0\leq k\leq 3$, 
the numbers $1$, $1$ are  $\binom{1}{k}$, $0\leq k\leq 1$.
The constants of normalization 
$\{c_k\}_{k=0}^2$ are determined by 
$2\beta h_\beta(0)=1$ for $1\leq\beta\leq6$, that is,
$c_0=1$, $c_1=-4$ and $c_2=3$. 

Using computer calculations 
we have verified Conjecture~\ref{Hform} for integer weight parameters $\gamma$ 
in the range $0\leq\gamma\leq 80$. This verification has been made 
similarly as was described for $F_\gamma$ above.

Our computer calculations have 
been carried out using the open source 
computer algebra package Maxima which is freely available under 
the GNU General Public License agreement 
(see http://maxima.sourceforge.net/index.shtml).

\section{Further results and comments}\label{furtherresults}

\setcounter{equation}{0}
\setcounter{thm}{0}
\setcounter{prop}{0}
\setcounter{lemma}{0}
\setcounter{cor}{0}
\setcounter{remark}{0}

Let us return to the 
Dirichlet problem (\ref{Dirichletproblem}).
Let  $w:\D\to(0,\infty)$ be a smooth radial weight function.
It is known that we have 
uniqueness of distributional solutions of 
(\ref{Dirichletproblem}) provided the weight function $w$ is area integrable, 
that is, $\int_0^1\tilde w(x)dx<\infty$ 
(see~\cite[Theorem~2.1]{Olofsson05}).
To ensure existence of a distributional solution $u$ of (\ref{Dirichletproblem}) 
for any given distributional boundary data $f_j\in\De'(\T)$ ($j=0,1$) 
we need in addition to area integrability 
of $w$ the assumption that
$$
\int_0^1 x^k\tilde w(x)dx\geq c(1+k)^{-N},\quad k\geq0,
$$
for some positive constants $c$ and $N$. 
This last condition on the moments of $w$ is satisfied if
$$
\inf_{0<x<1}\int_x^1\tilde w(x)dx/(1-x)^\alpha>0
$$ 
for some $\alpha>0$ 
(see~\cite[Section~3]{Olofsson05}).

We shall consider admissible radial weight functions $w:\D\to(0,\infty)$ 
such that the Poisson kernels 
$F_w$ and $H_w$ satisfy the $L^1$-bounds 
\begin{equation}\label{FL1bound}
\frac{1}{2\pi}\int_\T\vert F_w(re^{i\theta})\vert d\theta\leq C
\end{equation} 
and 
\begin{equation}\label{HL1bound}
\frac{1}{2\pi}\int_\T\vert H_w(re^{i\theta})\vert d\theta\leq C(1-r)
\end{equation} 
for $0\leq r<1$, where $C$ is an absolute constant.
Notice that \eqref{FL1bound} gives that 
$\lim_{r\to1}F_{w,r}=\delta_1$ in $\De'^0(\T)$ (distributions of order $0$) 
and similarly for \eqref{HL1bound}.
Notice also that these $L^1$-bounds \eqref{FL1bound} and \eqref{HL1bound} 
are satisfied for $w=w_\gamma$ if $F_\gamma$ and $H_\gamma$ 
verifies Conjectures~\ref{Fform} and~\ref{Hform} 
(see Section~\ref{boundarybehaviorbbfcns}). 

Recall that a homogeneous Banach space is a Banach space $B$ 
of distributions on $\T$ continuously embedded into $\De'(\T)$ 
such that for every $f\in B$ the translation (rotation)  
$$
\T\ni e^{i\tau}\mapsto f_{e^{i\tau}}\in B
$$
is a continuous $B$-valued map on $\T$ 
(see~\cite[Section~I.2]{Katznelson}). 
Here for $f\in L^1(\T)$ the translation $f_{e^{i\tau}}$ is defined by 
$f_{e^{i\tau}}(e^{i\theta})=f(e^{i(\theta-\tau)})$ for $e^{i\theta}\in\T$, 
and then extended to distributions in a standard way 
(see any text on distribution theory, for instance H\"ormander~\cite{Hormander}).

Recall that $\hat{f}(k)$ denotes the $k$-th Fourier coefficient of $f\in\De'(\T)$. 
In an earlier paper~\cite{Olofsson06} we have proved the following result. 

\begin{thm}\label{regularityDirichletproblem}
Let $w:\D\to(0,\infty)$ be an admissible radial weight function such that 
the Poisson kernels $F_w$ and $H_w$ satisfy~\eqref{FL1bound} 
and~\eqref{HL1bound}. Let $B_j$ ($j=0,1$) be two homogeneous 
Banach spaces such that
\begin{equation}\label{Blinking}
\{\kappa*f_0:\ f_0\in B_0\}=\{f_1\in B_1:\ \hat{f}_1(0)=0\},
\end{equation}
where $\kappa$ is as in Section~\ref{boundarybehaviorbbfcns}.
Let $f_j\in B_j$ ($j=0,1$) and 
let $u$ be the corresponding distributional solution of \eqref{Dirichletproblem}. 
Then the solution $u$ has the appropriate boundary values measured in the 
sense of convergence in the spaces $B_j$ ($j=0,1$):
$$
\lim_{r\to1}u_r=f_0\quad\text{in}\ B_0\quad\text{and}\quad
\lim_{r\to1}(u_r-f_0)/(1-r)=f_1\quad\text{in}\ B_1,
$$
where the $u_r$'s are as in \eqref{dfndilation} in the introduction.
\end{thm}

\begin{proof}
For full details of proof we refer to~\cite{Olofsson06};
the $L^1$-bounds \eqref{FL1bound} and \eqref{HL1bound} replace 
Lemmas~1.1 and~2.1 in that paper. 
\end{proof}

An example of a pair of homogeneous Banach spaces $B_j$ ($j=0,1$) satisfying 
the assumption \eqref{Blinking} 
in Theorem~\ref{regularityDirichletproblem} is provided by 
the Sobolev spaces 
$$
B_0=W^{m,p}(\T)\quad\text{and}\quad B_1=W^{m-1,p}(\T), 
$$
where $m\geq1$ is an integer and $1<p<\infty$.
Examples of spaces $B_j$ ($j=0,1$) satisfying 
\eqref{Blinking} can also be constructed using little oh 
Lipschitz/H\"older conditions on derivatives. 
We refer to~\cite{Olofsson06} for details.

We consider it a problem of interest to find conditions on 
a weight function $w:\D\to(0,\infty)$ which ensure the validity 
of the $L^1$-bounds \eqref{FL1bound} and \eqref{HL1bound}.

To this end we wish to mention also a somewhat related paper by Weir~\cite{W2} 
concerned with the derivation of a formula for the 
$w_1$-weighted biharmonic Green function 
which originates from Hedenmalm~\cite{Hprobl}.
Also, numerical studies of a phenomenon of 
extraneous zeros of Bergman kernel functions have been performed by Hedenmalm, 
Jakobsson and Perdomo (see~\cite[Section~4]{HP}).


\end{document}